\documentclass[10pt]{amsart}
\usepackage{cite,enumerate,setspace,comment}
\usepackage{MnSymbol}
\usepackage{paralist}
\usepackage{amsmath}
\usepackage{verbatim}
\usepackage{esint}
\usepackage{enumitem}
\theoremstyle{plain}
\newtheorem{theorem}{Theorem}[section]

\newtheorem{cor}[theorem]{Corollary}

\theoremstyle{definition}

\theoremstyle{remark}

\numberwithin{equation}{section}

\DeclareMathOperator{\diag}{diag}

\allowdisplaybreaks[1]
\begin{document}

\title [Fourier series of H\"{o}lder continuous functions] {Almost everywhere convergence of Fourier series on SU(2): the case of
H\"{o}lder continuous functions }

\author{David Grow and Donnie Myers}
\address {Missouri University of Science and Technology\\ Department of Mathematics and Statistics\\ Rolla, Missouri 65409-0020, USA.}

\let\oldthefootnote\thefootnote
\renewcommand{\thefootnote}{\fnsymbol{footnote}}
\footnotetext{Email addresses: {grow@mst.edu},{dfm@mst.edu}}
\let\thefootnote\oldthefootnote

\begin{abstract}
We consider an aspect of the open problem: Does every square-integrable function on $SU(2)$ have an almost
everywhere convergent Fourier series?  Let $0 < \alpha < 1$.  We show that to each countable set $E$ in
$SU(2)$ there corresponds an $\alpha$-H\"{o}lder continuous function on $SU(2)$ whose Fourier series diverges
on $E$.  We also show that the Fourier series of each $\alpha$-H\"{o}lder continuous function on $SU(2)$
converges almost everywhere.
\end{abstract}

\keywords{Fourier series, the two-dimensional special unitary group, H\"{o}lder continuous functions}

\subjclass[2010]{22E30, 43A50}

\date{2020-03-24}
\maketitle

\section{Introduction}

The Peter-Weyl theorem suggests the study of the formal Fourier series \\${\sum d_{\lambda}(\chi_{\lambda}\ast f)}$ of a
function $f$ on a compact, connected, semisimple Lie group $G$.  Here the sum is over the equivalence classes
of continuous irreducible unitary representations of $G$, $d_{\lambda}$ is the degree of the representation, and
$\chi_{\lambda}$ is its character.  The vast literature of Fourier analysis on $G$ is primarily concerned with mean
convergence or divergence of the Fourier series of $f \in L^p(G)$ (e.g.\cite{GST,GLZ,S,ST}), uniform or absolute
convergence of the partial sums if $f$ is smooth (e.g.\cite{C,GLZ,Ma1,DFM,MG,R1,R2,Sug,T}), almost everywhere convergence
or divergence of the partial sums if $f$ is a central function in $L^p(G)$ (e.g.\cite{CGTV,GT}),
and uniform, mean, or almost everywhere summability of the partial sums if $f$ belongs to various subspaces of $L^1(G)$
(e.g.\cite{CF,CGT,GLZ,Z}). The aim of this work is to advance the study of almost everywhere convergence or divergence of
Fourier partial sums of nonsmooth, possibly noncentral functions in $L^2(G)$.

Let $SU(2)$ denote the two-dimensional special unitary group.  We show that to each $\alpha$ in $(0,1)$ and to each
countable subset $E$ of $SU(2)$ there corresponds an $\alpha$-H\"{o}lder continuous function on $SU(2)$ whose
Fourier series diverges at each $x$ in $E$.  Since it is possible to arrange that such a set $E$ is dense in $SU(2)$,
the Fourier series of the corresponding function is divergent at infinitely many points in every nonempty open subset
of $SU(2)$.  Nevertheless, the Fourier series of each $\alpha$-H\"{o}lder continuous function on $SU(2)$
converges almost everywhere on $SU(2)$.

In fact, relying on a general almost everywhere convergence result of Dai \cite{D} for Fourier-Laplace series on
spheres, it follows that if $f$ in $L^2(SU(2))$ has an integral modulus of continuity $\Omega(f,t)$ satisfying
\begin{equation*}
\int_0^1\frac{\Omega^2(f,t)}{t}dt < \infty,
\end{equation*}
then the sequence of Fourier partial sums $\{S_{N}f(x)\}_{N=1}^{\infty}$ converges to $f(x)$ almost everywhere on $SU(2)$.
In particular, if $f$ is an $\alpha$-H\"{o}lder continuous function on $SU(2)$ for some $\alpha$ in $(0,1)$, or more generally
if $\Omega(f,t) = O(t^{\alpha})$ for some $\alpha$ in $(0,1)$, then the Fourier partial sums of $f$ converge to $f(x)$
almost everywhere on $SU(2)$.

It is still an open problem whether
\begin{equation}\label{aeconvergence}
  \underset{N\rightarrow\infty}{\lim}S_{N}f(x) = f(x)
\end{equation}
holds almost everywhere for every $f$ in $L^2(SU(2))$. It follows from a result of Pollard \cite{P} on Jacobi series that
if $f$ is a central function in $L^p(SU(2))$ for some $p > 4/3$, then (\ref{aeconvergence}) holds almost everywhere. On the
other hand, a general theorem of Stanton and Tomas \cite{ST} for compact, connected, semisimple Lie groups shows that if
$p < 2$ then there correspond an $f$ in $L^p(SU(2))$ and a subset $E$ of $SU(2)$ of full measure such that (\ref{aeconvergence})
fails for all $x$ in $E$.  Finally, the Peter-Weyl theorem implies that to each $f$ in $L^2(SU(2))$ there corresponds an
increasing sequence $\{N_j\}$ of positive integers such that
\begin{equation*}
  \underset{j\rightarrow\infty}{\lim}S_{N_j}f(x) = f(x)
\end{equation*}
for almost every $x$ in $SU(2)$. This lends some hope for a positive answer to the problem (\ref{aeconvergence}), as do
the results in this paper.

\section{Preliminaries}

Let $SU(2)$ denote the two-dimensional special unitary group. General matrices $x,y \in SU(2)$ can be expressed as
\begin{equation*}
x=\left[\begin{array}{cc}
\alpha_{1}+i\alpha_{2} & \beta_{1}+i\beta_{2}\\
-\left(\beta_{1}-i\beta_{2}\right) & \alpha_{1}-i\alpha_{2}
\end{array}\right], \quad
y=\left[\begin{array}{cc}
\gamma_{1}+i\gamma_{2} & \delta_{1}+i\delta_{2}\\
-\left(\delta_{1}-i\delta_{2}\right) & \gamma_{1}-i\gamma_{2}
\end{array}\right]
\end{equation*}
where $\alpha_{1},\alpha_{2},\beta_{1},\beta_{2},\gamma_{1},\gamma_{2},\delta_{1},\delta_{2}$ are real numbers satisfying
\begin{equation*}
\alpha_{1}^2+\alpha_{2}^2+\beta_{1}^2+\beta_{2}^2 = 1 = \gamma_{1}^2+\gamma_{2}^2+\delta_{1}^2+\delta_{2}^2.
\end{equation*}
Equip $SU(2)$ with the left and right translation invariant metric $d$ given by
\begin{align*}
d(x,y)=& \sqrt{\frac{1}{2}\textrm{tr}((x-y)(x-y)^{*})} \\
      =& \left((\alpha_1-\gamma_1)^2+(\alpha_2-\gamma_2)^2+(\beta_1-\delta_1)^2+(\beta_2-\delta_2)^2\right)^{1/2}. \notag
\end{align*}

Let $0 < \alpha < 1$ and let $f$ be a real function on $SU(2)$. If there exists a real number $M \geq 0$ such that
\begin{equation*}
\vert f(x)-f(y) \vert \leq Md^{\alpha}(x,y)
\end{equation*}
for all $x,y \in SU(2)$, then we say that $f$ is an $\alpha$-H\"{o}lder continuous function on $SU(2)$ and write
$f\in\textrm{Lip}_{\alpha}(SU(2))$.

Let $\mu$ denote normalized Haar measure on $SU(2)$. If $f$ and $g$ are Haar-integrable functions on $SU(2)$, then
their convolution product is defined for all $x \in SU(2)$ by
\begin{equation*}
(f \ast g)(x)=\int_{SU(2)}f(xy^{-1})g(y)d\mu(y).
\end{equation*}
If $f \in \textrm{L}^1(SU(2))$ and $f \ast g = g \ast f$ for all $g \in \textrm{L}^1(SU(2))$, then we call
$f$ a central function on $SU(2)$.  This is equivalent to the property that for $\mu$-almost every $x \in SU(2)$,
$f(yxy^{-1})=f(x)$ for all $y \in SU(2)$.  In particular, since every $x \in SU(2)$ is diagonalizable via a similarity transformation:
\begin{equation*}
yxy^{-1} =
  \begin{bmatrix}
    &e^{i\theta} &0 \\
    &0  &e^{-i\theta}
  \end{bmatrix}
  \equiv \omega(\theta)
\end{equation*}
where $y \in SU(2)$ and $e^{\pm i\theta}$ are the eigenvalues of $x$, it follows that if $f$ is central then for
$\mu$-almost every $x \in SU(2)$,
\begin{equation*}
f(x) = f(\omega(\theta))
\end{equation*}
where $\theta \in [0,\pi]$. Furthermore, we have
\begin{equation}\label{CentralIntegral}
  \int_{SU(2)}f(x)d\mu(x)=\frac{2}{\pi}\int_{0}^{\pi}f(\omega(\theta))\sin^2(\theta)d\theta
\end{equation}
when $f$ is central; this is a special case of equation (\ref{HaarIntegral}) below.

We denote the family of all (inequivalent) continuous, irreducible, unitary representations of $SU(2)$ by
$\{\pi_n\}_{n=0}^{\infty}$ (cf. pp. 125-136 in vol. 2 of \cite{HR}). Observe that $\pi_n$
has dimension $n+1$ and its character $\chi_n$ is the continuous central function on $SU(2)$ given by
\begin{equation*}
\chi_n(x) = \textrm{trace}\left(\pi_n(x)\right).
\end{equation*}
If $e^{\pm i\theta}$ are the eigenvalues of $x$, then
\begin{equation*}
\chi_n(x) = \chi_n(\omega(\theta)) = \frac{\sin((n+1)\theta)}{\sin(\theta)};
\end{equation*}
the rightmost member this identity is the $n^{\text{th}}$ Chebyshev polynomial of the second kind on $[0,\pi]$.

The Dirichlet kernel $\{\mathbf{D}_N\}_{N=0}^{\infty}$ on $SU(2)$ is the sequence of continuous central
functions given by
\begin{equation*}
\mathbf{D}_N(x)=\sum_{n=0}^N (n+1)\chi_n(x).
\end{equation*}
The $N$th Fourier partial sum of a function $f \in \textrm{L}^1(SU(2)$ is the continuous function on
$SU(2)$ given by
\begin{equation*}
S_Nf(x) = (f \ast \mathbf{D}_N)(x).
\end{equation*}

\section{Divergence of Fourier partial sums on a countable subset}

\begin{theorem}\label{divergence}
Let $\alpha\in(0,1)$ and let $\{x_{i}\}_{i=1}^{\infty}$ be any countable subset of $SU(2)$. Then there exists a function
$f\in\textrm{Lip}_{\alpha}(SU(2))$ such that
\[
\underset{N\geq1}{\sup}|S_{N}f(x_{i})|=\infty
\]
for all $i=1,2,3,\ldots$.
\end{theorem}

\begin{proof}
Recall that for $\alpha \in (0,1)$, $\textrm{Lip}_{\alpha}(SU(2))$ is a Banach space with norm
\begin{equation*}
\Vert f\Vert_{\textrm{Lip}_{\alpha}(SU(2))}=\underset{x\in SU(2)}{\sup}|f(x)|+\underset{\underset{x\neq y}{x,y\in}SU(2)}{\sup}\frac{|f(x)-f(y)|}{d^{\alpha}(x,y)}.
\end{equation*}
Fix $x \in SU(2)$ and $n \in \mathbb{N}$, and set $\varLambda_{n}^{x}(f)=S_{n}f(x)$.  Each $\varLambda_{n}^{x}$ is a bounded linear
functional on $\textrm{Lip}_{\alpha}(SU(2))$ of norm
\begin{equation*}
  \Vert\varLambda_{n}^{x}\Vert =\sup\left\{ |S_{n}f(x)|: f\in\textrm{Lip}_{\alpha}(SU(2)),\;\Vert f\Vert_{\textrm{Lip}_{\alpha}(SU(2))}\leq1\right\} \leq \Vert\boldsymbol{D}_{n}\Vert_{L^{1}(SU(2))}.
\end{equation*}
Specializing to the case when $x = e$, the identity matrix in $SU(2)$, we have
\begin{equation*}
  \varLambda_{n}^{e}(f)=(f\ast\boldsymbol{D}_{n})(e)=
      \int_{SU(2)}f(y)\boldsymbol{D}_{n}(y)d\mu(y)=
      \frac{2}{\pi}\int_{0}^{\pi}f(\omega(\theta))\boldsymbol{D}_{n}(\omega(\theta))\sin^2(\theta)d\theta
\end{equation*}
by (\ref{CentralIntegral}). Note that
\begin{equation*}
  \boldsymbol{D}_{n}(\omega(\theta))=\frac{-1}{2\sin(\theta)}D_{n+1}^{\prime}(\theta)
\end{equation*}
where
\begin{equation*}
  D_n(t)=\sum_{j=-n}^{n}e^{ijt}=1+2\sum_{j=1}^{n}\cos(jt)=\frac{\sin((2n+1)t/2)}{\sin(t/2)}
\end{equation*}
is the Dirichlet kernel on $[-\pi,\pi]$.  It follows that
\begin{align*}
\varLambda_{n}^{e}(f)
  =&\frac{1}{\pi}\int_{0}^{\pi}f(\omega(\theta))\cos^{2}\left(\frac{\theta}{2}\right)D_{n+1}(\theta)d\theta\\
   &\quad - \frac{(2n+3)}{\pi}\int_{0}^{\pi}f(\omega(\theta))\cos\left(\left(n+\frac{3}{2}\right)\theta\right)\cos\left(\frac{\theta}{2}\right)d\theta.
\end{align*}

The absolute maxima and minima of the function $h_n(\theta)=\cos\left(\left(n+\frac{3}{2}\right)\theta\right)$ on $[0,\pi]$
occur at the endpoints of the intervals $I_{k}=\left[\frac{2k\pi}{2n+3},\frac{2(k+1)\pi}{2n+3}\right]$ where $k \in \{0,1,2,...,n\}$.
Let $g_n$ be the sawtooth function on $[0,\pi]$ determined by $g_n\left(\frac{2k\pi}{2n+3}\right) = (-1)^k$ for $0\leq k \leq n+1$,
$g_n(\pi)=0$, and $g_n$ is piecewise linear between these points.  Define a central function $f_n$ on $SU(2)$ by
$f_n(\omega(\theta)) = g_n(\theta)$ for $\theta \in [0,\pi]$. It is easy to see that $f_n$ belongs to $\textrm{Lip}_{\alpha}(SU(2))$;
in fact,
\begin{equation*}
  \frac{\vert f_n(x)-f_n(y) \vert}{d^{\alpha}(x,y)} \leq \left(\frac{\pi}{2}\right)^{\alpha}\left(\frac{2\pi}{2n+3}\right)^{1-\alpha}
    \leq \pi
\end{equation*}
for all distinct matrices $x$ and $y$ in $SU(2)$.

Since $g_n(\theta)\cos\left(\left(n+\frac{3}{2}\right)\theta\right) \geq g_n^2(\theta) \geq 0$ on each interval $I_k$ and on
$\left[\frac{2(n+1)\pi}{2n+3},\pi\right]$, and since the function $\theta \mapsto \cos(\theta/2)$ is positive and decreasing
on $[0,\pi)$, it follows that
\begin{align*}
   \int_{I_k}f_n(\omega(\theta))\cos\left(\left(n+\frac{3}{2}\right)\theta\right)\cos(\theta/2)d\theta
        &\geq \int_{I_k}g_n^2(\theta)\cos(\theta)d\theta \\
        &\geq \cos\left(\frac{(k+1)\pi}{2n+3}\right)\int_{I_k}g_n^2(\theta)d\theta \\
        &= \frac{2\pi}{3(2n+3)}\cos\left(\frac{(k+1)\pi}{2n+3}\right)
\end{align*}
for all $k \in \{0,1,2,...,n\}$ and
\begin{equation*}
\int_{\frac{2(n+1)\pi}{2n+3}}^{\pi}f_n(\omega(\theta))\cos\left(\left(n+\frac{3}{2}\right)\theta\right)\cos(\theta/2)d\theta \geq 0.
\end{equation*}
Adding these inequalities we obtain
\begin{align*}
\intop_{0}^{\pi}f_n(\omega(\theta))\cos\left(\left(n+\frac{3}{2}\right)\theta\right)\cos\left(\frac{\theta}{2}\right)d\theta
    & \geq \frac{2}{3}\left(\frac{\pi}{2n+3}\right)\sum_{k=1}^{n+1}\cos\left(\frac{k\pi}{2n+3}\right)\\
    & = \frac{2}{3}\left(\frac{\pi}{2n+3}\right)\left\{ D_{n+1}\left(\frac{\pi}{2n+3}\right)-1\right\},
\end{align*}
and hence
\begin{equation*}
\biggl|\frac{(2n+3)}{\pi}\int_{0}^{\pi}f_n(\omega(\theta))\cos\left(\left(n+\frac{3}{2}\right)\theta\right)\cos\left(\frac{\theta}{2}\right)d\theta\biggr|
\geq\frac{2}{3}\left\{ D_{n+1}\left(\frac{\pi}{2n+3}\right)-1\right\} .
\end{equation*}
Since the function $\theta\mapsto f_n(\omega(\theta))\cos^{2}(\theta/2)$ is uniformly bounded by $1$ on $[0,\pi]$,
\begin{align*}
\biggl |\frac{1}{\pi}\int_{0}^{\pi}f_n(\omega(\theta))\cos^{2}(\theta/2)D_{n+1}(\theta)d\theta \biggr|
    & \leq \frac{1}{\pi}\intop_{0}^{\pi}|D_{n+1}(\theta)|d\theta \\
    & = \frac{4}{\pi^{2}}\log(n+1)+o(1)
\end{align*}
as $n\rightarrow\infty$. Consequently
\begin{equation*}
\frac{|\varLambda_{n}^{e}(f_n)|}{\Vert f_n\Vert_{\textrm{Lip}_{\alpha}(SU(2))}}
    \geq \frac{\frac{2}{3}\left\{ D_{n+1}\left(\frac{\pi}{2n+3}\right)-1\right\}
        -\left(\frac{4}{\pi^{2}}\log(n+1)+o(1)\right)}{1+\pi}.
\end{equation*}
But $D_{n+1}\left(\frac{\pi}{2n+3}\right) = \left(\sin\left(\frac{\pi}{2(2n+3)}\right)\right)^{-1} \geq \frac{2(2n+3)}{\pi}$
and hence
\begin{equation*}
\Vert\varLambda_{n}^{e}\Vert =
\sup\left\{\frac{|\varLambda_{n}^{e}(f)|}{\Vert f\Vert_{\textrm{Lip}_{\alpha}(SU(2))}}:\quad f\in\textrm{Lip}_{\alpha}(SU(2)),
    \quad f\neq 0 \right\}
\end{equation*}
is asymptotically bounded below by
\begin{equation*}
    \frac{\frac{2}{\pi}(2n+3)-\frac{4}{\pi^{2}}\log(n+1)}{1+\pi}
\end{equation*}
as $n\rightarrow\infty$.  Consequently, the sequence of bounded linear functionals
\begin{equation*}
    \varLambda_{n}^{e}(f)=S_{n}f(e)
\end{equation*}
is not uniformly bounded on the Banach space $\textrm{Lip}_{\alpha}(SU(2))$ as $n\rightarrow\infty$. By the uniform
boundedness principle
\begin{equation*}
    \underset{n\geq1}{\sup}|S_{n}f(e)| = \infty
\end{equation*}
for all $f$ belonging to some dense $G_{\delta}$ set in $\textrm{Lip}_{\alpha}(SU(2))$.

If $z \in SU(2)$, define the left translation operator $L_z$ on $\textrm{Lip}_{\alpha}(SU(2))$ by $L_{z}f(y) = f(zy)$
for all $y \in SU(2)$.  For each element of the countable subset $\{x_{i}\}_{i=1}^{\infty}$ of $SU(2)$, observe that
\begin{equation*}
\frac{|\varLambda_{n}^{x_i}\left(L_{{x_i}^{-1}}f_{n}\right)|}{\Vert L_{{x_i}^{-1}}f_{n}\Vert_{\textrm{Lip}_{\alpha}(SU(2))}}
  =\frac{|\varLambda_{n}^{e}\left(f_{n}\right)|}{\Vert f_{n}\Vert_{\textrm{Lip}_{\alpha}(SU(2))}},
\end{equation*}
so there corresponds a dense $G_{\delta}$ subset $E_{x_i}$ of $\textrm{Lip}_{\alpha}(SU(2))$ such that
\begin{equation*}
    \underset{n\geq1}{\sup}|S_{n}f(x_i)|=\infty
\end{equation*}
for all $f \in E_{x_i}$. By the Baire category theorem $E=\bigcap_{n=1}^{\infty}E_{x_{i}}$ is dense in
$\textrm{Lip}_{\alpha}(SU(2))$. In particular, $E$ is nonempty and any $f \in E$ gives the desired conclusion.
\end{proof}

The problem of pointwise convergence for the Fourier series of central functions in $\textrm{Lip}_{\alpha}(SU(2))$
for some $\alpha \in (0,1)$ is related to an analogous problem for Fourier-Jacobi series of functions
$f:[-1,1] \rightarrow \mathbb{R}$.  Recall that the Chebyshev polynomials of the second kind $U_n (n=0,1,2,...)$
are a special case of the Jacobi polynomials $P_n^{(\alpha,\beta)}$ with $\alpha = \beta = 1/2$; i.e.
\begin{equation*}
  U_n(\cos(\theta))= P_n^{(1/2,1/2)}(\cos(\theta)) = \frac{\sin((n+1)\theta)}{\sin(\theta)} \quad (n=0,1,2,...;\theta \in [0,\pi]).
\end{equation*}
Clearly $\{U_n\}_{n=0}^{\infty}$ is an orthogonal set of functions with respect to the inner product
\begin{equation*}
  \langle F,G \rangle = \int_{-1}^{1}F(t)G(t)\sqrt{1-t^2}dt.
\end{equation*}
If $F \in L^2([-1,1],\sqrt{1-t^2}dt)$, let $s_N(F;t)$ denote the $N$th partial sum of the Fourier series of $F$
with respect to $\{U_n\}_{n=0}^{\infty}$.  If $F$ has a modulus of continuity $\omega(F,h)$ satisfying the Dini-Lipschitz
condition
\begin{equation*}
  \underset{h\rightarrow 0^{+}}{\lim}\left(\omega(F,h)\log(h^{-1})\right)= 0,
\end{equation*}
then a classical theorem \cite{Sue} assures $s_N(F;t) \rightarrow F(t)$ uniformly on any interval $[-1+\delta, 1-\delta]$
with $\delta \in (0,1)$.  Furthermore, a general theorem of Belen'kii \cite{B} guarantees that if $F$ satisfies a
Dini-Lipschitz condition and $\{s_N^{(\alpha,\beta)}(F;\pm 1)\}_{N=0}^{\infty}$ converge for some $\alpha > -1$ and
$\beta > -1$, then the Fourier-Jacobi series of $F$ converges uniformly to $F$ on $[-1,1]$.  Observe that a central
function $f \in L^2(SU(2))$ corresponds to a function $F \in L^2([-1,1],\sqrt{1-t^2}dt)$ via $F(\cos(\theta)) = f(x)$
where $e^{i\theta},e^{-i\theta}$ are the eigenvalues of $x \in SU(2)$.  It is easy to check that the Fourier partial
sums of $F$ and $f$ satisfy the identity
\begin{equation*}
  s_N(F;\cos(\theta)) = (S_Nf)(x) \quad (N=0,1,2,...),
\end{equation*}
and if $f \in \textrm{Lip}_{\alpha}(SU(2))$ for some $\alpha \in (0,1)$ then $F$ satisfies a Dini-Lipschitz condition.
This leads to the following result.
\begin{theorem}\label{central}
  Let $f$ be a central function in $\textrm{Lip}_{\alpha}(SU(2))$ for some $\alpha \in (0,1)$.  Then:
  \begin{enumerate}[label=(\alph*)]
    \item $S_Nf(x) \rightarrow f(x)$ uniformly outside any open set containing $\{e,-e\}$;
    \item $S_Nf(x) \rightarrow f(x)$ uniformly on $SU(2)$ if $\{S_Nf(\pm e)\}_{N=0}^{\infty}$ converge.
  \end{enumerate}
\end{theorem}
The authors of this paper originally wondered if they could delete the word ``central'' from the hypothesis of Theorem
\ref{central} and still obtain conclusions (a) and (b).  Theorem \ref{divergence} shows that there is no
possibility of such an analogue of Theorem \ref{central} for general functions in $\textrm{Lip}_{\alpha}(SU(2))$ for
$0 < \alpha < 1$.  This is so because the points of divergence for the Fourier partial sums of such a noncentral
function need no longer be at the ``poles'' $\pm e$. According to Theorem \ref{divergence}, points of divergence
for $\textrm{Lip}_{\alpha}(SU(2))$ functions can be dense in $SU(2)$.

\section{Almost everywhere convergence of Fourier partial sums}

Note that $SU(2)$ is isometrically homeomorphic to the unit sphere $S^3$ in $\mathbb{R}^4$ via the isometry $\eta$
from $SU(2)$ onto $S^3$ given by
\begin{equation*}
\left[\begin{array}{cc}
\alpha_{1}+i\alpha_{2} & \beta_{1}+i\beta_{2}\\
-\beta_{1}+i\beta_{2} & \alpha_{1}-i\alpha_{2}
\end{array}\right] \mapsto \left(\alpha_1,\alpha_2,\beta_1,\beta_2\right).
\end{equation*}
Furthermore, the spherical coordinate system on $S^3$:
\begin{equation*}
  \alpha_{1}= \cos(\theta),\alpha_{2}= \sin(\theta)\cos(\phi),\beta_{1}= \sin(\theta)\sin(\phi)\cos(\psi),
  \beta_{2}= \sin(\theta)\sin(\phi)\sin(\psi),
\end{equation*}
where $\phi \in [0,\pi]$, $\theta \in [0,\pi]$, and $\psi \in [0,2\pi]$,
can be transferred to $SU(2)$ and forms a convenient parametrization thereof:
\begin{equation*}
x(\phi,\theta,\psi)=\left[\begin{array}{cc}
\cos(\theta)+i\sin(\theta)\cos(\phi) & \sin(\theta)\sin(\phi)e^{i\psi}\\
-\sin(\theta)\sin(\phi)e^{-i\psi} & \cos(\theta)-i\sin(\theta)\cos(\phi)
\end{array}\right].
\end{equation*}
We denote by $\Phi$ the mapping from $[0,\pi]\times[0,\pi]\times[0,2\pi]$ onto $SU(2)$ given by
$(\phi,\theta,\psi) \mapsto x(\phi,\theta,\psi)$.

In spherical coordinates, normalized Haar measure $\mu$ on $SU(2)$ satisfies
\begin{equation}\label{HaarIntegral}
\int_{SU(2)}f(x)d\mu(x) =
    \frac{1}{2\pi^2}\int_0^{\pi}\int_0^{\pi}\int_0^{2\pi}
       (f\circ\Phi)(\phi,\theta,\psi))\sin^2(\theta)\sin(\phi)d\psi d\phi d\theta
\end{equation}
for every $f \in \textrm{L}^1(SU(2),\mu) = \textrm{L}^1(SU(2))$ (cf. pp. 133-134 in vol. 2 of \cite{HR}).
Using spherical coordinates, the $N$th Fourier partial sum of a function $f \in \textrm{L}^1(SU(2))$
can be written as
\begin{align*}
(&S_{N}f)(x(\phi_{0},\theta_{0},\psi_{0}))
    =\int_{SU(2)}\mathbf{D}_{N}(y)f(x(\phi_{0},\theta_{0},\psi_{0})y^{-1}(\phi,\theta,\psi))d\mu(y) \\
 & \nonumber
    =\frac{-1}{4\pi^{2}}\int_{0}^{\pi}\int_{0}^{\pi}\int_{0}^{2\pi}D_{N+1}^{\prime}(\theta)
        f(x(\phi_{0},\theta_{0},\psi_{0})y^{-1}(\phi,\theta,\psi))\sin(\theta)\sin(\phi)d\psi d\phi d\theta\\
 &  \nonumber
    =\frac{-1}{\pi}\intop_{0}^{\pi}D_{N+1}^{\prime}(\theta)\sin(\theta)[Q_{x}f](\theta)d\theta,
\end{align*}
where
\begin{equation*}
[Q_{x}f](\theta)=\frac{1}{4\pi}\int_{0}^{\pi}\int_{0}^{2\pi}
    f(x(\phi_{0},\theta_{0},\psi_{0})y^{-1}(\phi,\theta,\psi))\sin(\phi)d\psi d\phi.
\end{equation*}
The above identities for the Fourier partial sums of $f \in L^p(SU(2))$ reduce
convergence problems at a point $x$ in the three-dimensional manifold $SU(2)$ to the behavior of the real function
$[Q_{x}f]$ on the interval $[0,\pi]$. See \cite{DFM} and \cite{MG} for applications of this principle.

Because $SU(2)$ is isometrically homeomorphic to the unit sphere $S^3$, the theory of Fourier series
of functions on $SU(2)$ is closely connected with that of Fourier-Laplace series on $S^3$.  Specifically, if $\sigma_3$
denotes normalized surface measure on $S^3$, then the $N$th partial sum of the Fourier-Laplace series of
$F \in L^2(S^3,\sigma_3)$ is equal to the $N$th partial sum of the Fourier series of $f = F\circ\eta \in L^2(SU(2))$
\cite[pp. 93-94]{DFM}.  Moreover, it follows from (\ref{HaarIntegral}) that $\sigma_{3}(E) = \mu(\eta^{-1}(E))$ for all
spherical boxes
\begin{equation*}
  E = \{\overrightarrow{r}(\phi,\theta,\psi) \in S^3 : (\phi,\theta,\psi) \in [\alpha,\beta]\times[\gamma,\delta]\times[\epsilon,\nu]\}
\end{equation*}
in $S^3$; in particular, sets of zero Haar measure in $SU(2)$ correspond to sets of zero surface measure in $S^3$.

More generally, let $S^n$ denote the unit sphere in $\mathbb{R}^{n+1}$.  For $F \in L^2(S^n,\sigma_n) = L^2(S^n)$ and
$\theta \in (0,\pi)$, define the spherical translation operator $T_{\theta}$ by
\begin{equation*}
(T_{\theta}F)(\overrightarrow{x})=\frac{1}{\vert{S^{n-1}}\vert}\int_{\{\overrightarrow{y}\in S^n :
\overrightarrow{x}^T\overrightarrow{y} = 0 \}}F(\overrightarrow{x}\cos(\theta)+\overrightarrow{y}\sin(\theta))d\nu(\overrightarrow{y})
\end{equation*}
where $\nu$ denotes surface measure on the $n-1$ dimensional sphere
$\{\overrightarrow{y} \in S^n : \overrightarrow{x}^T\overrightarrow{y} = 0\}$.  Let
\begin{equation*}
(\Delta_{\theta}F)(\overrightarrow{x})= F(\overrightarrow{x})-(T_{\theta}F)(\overrightarrow{x})
\end{equation*}
denote the difference operator acting on $F \in L^2(S^n)$, and for $t>0$ let
\begin{equation*}
  \omega(F,t)=\sup\{\Vert\Delta_{\theta}F\Vert_{L^2(S^n)}:0 < \theta \leq t\}
\end{equation*}
denote the integral modulus of continuity of $F \in L^2(S^n)$. The following result is the $r=2$ case of Theorem 2
obtained by Feng Dai in \cite{D}.

\begin{theorem}\label{Dai}
Let $F \in L^2(S^n)$. If
\begin{equation*}
\int_0^1\frac{\omega^2(F,t)}{t}dt < \infty
\end{equation*}
then the Fourier-Laplace partial sums of $F$ at $\overrightarrow{x}$ converge to $F(\overrightarrow{x})$ $\sigma_n$-almost everywhere.
\end{theorem}

Let us specialize to the case $n=3$ of the previous theorem. If $\overrightarrow{x}=\eta(x(\phi_0,\theta_0,\psi_0))$ then
\begin{align*}
(T_{\theta}F)(\overrightarrow{x}) &
=\frac{1}{4\pi}\int_{0}^{\pi}\int_{0}^{2\pi}F(O(\cos(\theta)\overrightarrow{e_1}+\sin(\theta)\overrightarrow{v}))\sin(\phi)d\psi d\phi \\
& = \frac{1}{4\pi}\int_{0}^{\pi}\int_{0}^{2\pi}F(\cos(\theta)\overrightarrow{x}+\sin(\theta)O\overrightarrow{v})\sin(\phi)d\psi d\phi
\end{align*}
where $\overrightarrow{e_1} = [1,0,0,0]^T$, $\overrightarrow{v}=[0,\cos(\phi),\sin(\phi)\cos(\psi),\sin(\phi)\sin(\psi)]^T$, and
\begin{equation*}
  O = \cos(\theta_0)\diag(1,-1,-1,-1) + \sin(\theta_0)S
\end{equation*}
with
\begin{equation*}
  S =
  \begin{bmatrix}
    0 & \cos(\phi_0) & \sin(\phi_0)\cos(\psi_0) & \sin(\phi_0)\sin(\psi_0) \\
    \cos(\phi_0) & 0 & \sin(\phi_0)\sin(\psi_0) & -\sin(\phi_0)\cos(\psi_0) \\
    \sin(\phi_0)\cos(\psi_0) & -\sin(\phi_0)\sin(\psi_0) & 0 & \cos(\phi_0) \\
    \sin(\phi_0)\sin(\psi_0) & \sin(\phi_0)\cos(\psi_0) & -\cos(\phi_0) & 0
  \end{bmatrix}.
\end{equation*}
If $F = f\circ\eta^{-1}$ then a routine computation yields
\begin{equation*}
(T_{\theta}F)(\overrightarrow{x})=\frac{1}{4\pi}\int_{0}^{\pi}\int_{0}^{2\pi}f(x(\phi_0,\theta_0,\psi_0)y^{-1}(\phi,\theta,\psi))
    \sin(\phi)d\psi d\phi
\end{equation*}
and hence it follows that
\begin{equation}\label{SphericalTranslationIdentity}
(T_{\theta}F)(\overrightarrow{x}) = [Q_{x}f](\theta)
\end{equation}
for all $0 < \theta < \pi$.

Let $\theta \in [0,\pi]$. Define the difference operator $\delta_{\theta}$ on $L^2(SU(2))$ by
\begin{equation*}
  \delta_{\theta}(f)(x) = f(x) - [Q_{x}f](\theta)
\end{equation*}
and let
\begin{equation*}
  \Omega(f,t)=\sup\{\Vert\delta_{\theta}f\Vert_{L^2(SU(2))}:0 < \theta \leq t\}
\end{equation*}
denote the integral modulus of continuity of $f \in L^2(SU(2))$. Using Theorem \ref{Dai} and identity
(\ref{SphericalTranslationIdentity}) yields the following almost everywhere convergence result on $SU(2)$.
\begin{theorem}
Let $f \in L^2(SU(2))$.  If
\begin{equation*}
\int_0^1\frac{\Omega^2(f,t)}{t}dt < \infty
\end{equation*}
then the Fourier partial sums of $f$ at $x$ converge to $f(x)$ $\mu$-almost everywhere on $SU(2)$.
\end{theorem}

The next three corollaries are immediate consequences.
\begin{cor}
Let $f \in L^2(SU(2))$.  If $\Omega(f,t)=O(\frac{1}{\log^{\beta}(1/t)})$ for some $\beta > 1/2$ then
$S_{N}f(x) \rightarrow f(x)$ for almost every $x \in SU(2)$.
\end{cor}
\begin{cor}
Let $f \in L^2(SU(2))$.  If $\Omega(f,t)=O\left(t^{\alpha}\right)$ for some $\alpha \in (0,1)$ then
$S_{N}f(x) \rightarrow f(x)$ for almost every $x \in SU(2)$.
\end{cor}
\begin{cor}
Let $f \in \textrm{Lip}_{\alpha}(SU(2))$ for some $\alpha \in (0,1)$.  Then $S_{N}f(x) \rightarrow f(x)$ for
almost every $x \in SU(2)$.
\end{cor}

\end{document}